\documentclass{amsart}
\usepackage{amsthm,amsmath,amsfonts,amssymb}
\usepackage[all]{xy}
\usepackage{color}
\usepackage[pagebackref, colorlinks=true, linkcolor=blue, citecolor=blue]{hyperref}
\usepackage{tikz-cd}

\makeatletter  
\@namedef{subjclassname@2020}{%
	\textup{2020} Mathematics Subject Classification}
\makeatother

\numberwithin{equation}{section}
\newtheorem{theorem}{Theorem}[section]
\newtheorem{lemma}[theorem]{Lemma}
\newtheorem{proposition}[theorem]{Proposition}

\theoremstyle{definition}
\newtheorem{remark}[theorem]{Remark}
\newtheorem{definition}[theorem]{Definition}

\numberwithin{equation}{section}

\newcommand{\sC}{\mathcal{C}}
\newcommand{\sF}{\mathcal{F}}
\newcommand{\colim}{\operatorname{colim}}
\newcommand{\K}{\mathbb{K}}
\newcommand{\Z}{\mathbb{Z}}
\newcommand{\N}{\mathbb{N}}

\newcommand{\Li}{\mathbb{L}}
\renewcommand{\tilde}{\widetilde}

\newcommand{\ro}{{\rho}}
\newcommand{\tw}{\operatorname{tw}}

\newcommand{\bM}{\mathbf{M}}

\newcommand{\sW}{\mathcal{W}}
\newcommand{\Der}{\operatorname{Der}}

\newcommand{\Hom}{\operatorname{Hom}}

\newcommand{\Span}{\operatorname{Span}}

\newcommand{\Id}{\operatorname{Id}}

\title{An elementary approach to the model structure on DG-Lie algebras}

\author{Emma Lepri}
\address{\newline
	Universit\`a degli studi di Roma La Sapienza,\hfill\newline
	Dipartimento di Matematica  Guido
	Castelnuovo,\hfill\newline
	P.le Aldo Moro 5,
	I-00185 Roma, Italy.\medskip}
\email{emma.lepri@uniroma1.it}

\subjclass[2020]{18N40, 17B70, 	16T15}
\keywords{DG-Lie algebras, model categories}
	
\begin{document}
		
		\maketitle

\begin{abstract}
	This paper contains an elementary proof of the existence of the classical model structure on the category of unbounded DG-Lie algebras over a field of characteristic zero, with an emphasis on the properties of free and semifree extensions, which are particularly nice cofibrations. 
	The cobar construction of a locally conilpotent cocommutative coalgebra is shown to be an example of semifree DG-Lie algebra. 
		We also give an example of a non-cofibrant DG-Lie algebra whose underlying graded Lie algebra is free; this cannot occur in the bounded above case, where DG-Lie algebras of this form are always cofibrant.
\end{abstract}

\section{Introduction}
A model structure on the category of bounded above\footnote{To be precise, the category considered by Quillen is of DG-Lie algebras living in strictly positive degrees using the homological convention, however in this paper the cohomological convention is used.} DG-Lie algebras over the rational numbers was constructed by Quillen in \cite{QuiR}. Denoting by $\mathbf{DGLA}^{< 0}_{\mathbb{Q}}$ the category of DG-Lie algebras over $\mathbb{Q}$ living in degrees $ i < 0$, the weak equivalences in this model structure are quasi-isomorphisms, the fibrations are 
the maps surjective in degrees $ i < -1$, and the cofibrations are the maps with the left lifting property with respect to the trivial fibrations. This model structure was extended by Neisendorfer in \cite{Nei} to the category of DG-Lie algebras  living in degrees $i \leq 0$ over a field of characteristic zero.

In \cite{Hin}, Hinich constructed model structures on the categories of algebras over certain operads in chain complexes satisfying a {splitness} condition. In particular, he constructed a model structure on the category of unbounded DG-Lie algebras over a ring $k$ containing $\mathbb{Q}$. In this model structure the weak equivalences are again the quasi-isomorphisms, the fibrations are the surjective maps, and the cofibrations are the maps with the left lifting property with respect to the trivial fibrations. Over a field of characteristic zero, the existence of this model structure was also proved in \cite{DAGX} by Lurie, who furthermore showed that the model category is left proper and combinatorial. 

The purpose of this expository paper is to give a direct, elementary proof of the existence of the model structure on the category of unbounded DG-Lie algebras over a field of characteristic zero, and to give an explicit construction of a subclass of the cofibrations. The proof is based on the properties of free and semifree extensions of DG-Lie algebras.

In particular, we will prove that the cofibrations in the model structure are the retracts of semifree extensions of DG-Lie algebras, and that trivial cofibrations are the retracts of free extensions. 
In \cite{QuiR}, Quillen introduced a definition of free map of DG-Lie algebras, which does not involve the differentials, and proved that the cofibrations in the bounded above case are  exactly the retracts of the free maps. This implies that in the bounded above case every DG-Lie algebra whose  underlying graded Lie algebra is free is cofibrant. Quillen's definition of free maps does not coincide with our definition of semifree extensions; however in the bounded above case every free map is also a semifree extension. In the unbounded case this is no longer true: in general Quillen's free maps
are not semifree extensions according to our definition. In fact, we show that in the unbounded case DG-Lie algebras whose underlying graded Lie algebra is free are not always cofibrant.

\emph{Organisation of the paper:} In Section~\ref{sec.L}, the basics about free DG-Lie algebras are recalled. Section~\ref{sec.contract} contains some lemmas about the extension of a contraction of chain complexes to a contraction of tensor algebras and free DG-Lie algebras. In Section~\ref{sec.model} we recall the definition of model structure, and the fact that a left pre-model structure naturally gives rise to a model structure. The proof of the existence of the classical model structure on the the category of DG-Lie algebras is given in Section~\ref{sec.modeldgla}, using left pre-model structures. 
Finally, Section~\ref{sec.ex} is devoted to some examples: in particular we show that, unlike in the bounded case, in the unbounded case
DG-Lie algebras which are free as graded Lie algebras are not always cofibrant, and that
the cobar construction of a locally conilpotent cocommutative coalgebra is a semifree DG-Lie algebra. 
\subsection*{Notation and setup}

Throughout the paper, $\K$ will denote a fixed field of characteristic 0.
Every (graded) vector space is assumed over $\K$ and the symbol $\otimes$ denotes the tensor product over $\K$, unless otherwise specified.
If $V=\oplus_{n\in\Z}V^n$ is a graded vector space, $\bar{a}$ denotes the degree of a non-zero homogeneous element $a$: in other words $\bar{a}=n$ whenever $a\not=0$ and $a\in V^n$. 

In a complex of vector spaces, the differential will always have degree $+1$.
As usual, for every complex of vector spaces $V$, we shall denote by $Z^n(V), B^n(V)$ and $H^n(V)$ the space of $n$-cocycles, the space of $n$-coboundaries and the $n$th cohomology group, respectively.
The vector space generated by $v$ will be denoted by $\K\langle v \rangle$.

\section{Free DG-Lie algebras}\label{sec.L}

In this section we recall the construction of the free DG-Lie algebra generated by a complex of vector spaces over a field of characteristic zero. This construction is similar to the free Lie algebra generated by a non-graded vector space over a field of characteristic zero; the only additional thing to prove is that the differential of the complex of vector spaces induces one on the free graded Lie algebra. The references for this section are \cite[Chapter 21]{FHT} and \cite[Appendix B]{QuiR}.

Recall that the free associative DG-algebra generated by a complex of vector spaces $(V,d)$ is the tensor algebra $T(V)$. As a graded vector space it is given by
$T(V)= \bigoplus_{n \geq 0} V^{\otimes  n}$,
with product defined as the bilinear extension of \[(v_1 \otimes \cdots \otimes v_a)(w_1 \otimes \cdots \otimes w_b ) = v_1 \otimes \cdots \otimes  v_a \otimes w_1 \otimes \cdots \otimes w_b,\]
and differential 
\[ d(v_1\otimes \cdots \otimes v_s) = \sum^{s}_{i=1} (-1)^{\overline{v_1}+ \cdots \overline{v_{i-1}}} v_1 \otimes \cdots \otimes dv_i \otimes \cdots \otimes v_s.\]

An associative DG-algebra $A$ can be considered as a DG-Lie algebra, with bracket induced by the graded commutator $[a,b]= ab - (-1)^{\overline{a}\overline{b}}ba$. To avoid confusion we denote by $A_L$ the DG-algebra $A$ seen as a DG-Lie algebra. In particular, given a graded vector space $V$, we denote by $T(V)_L$ the graded Lie algebra with bracket defined as the bilinear extension of $[v,w]= v\otimes w - (-1)^{\overline{v} \ \overline{w}}w\otimes v$.

The free graded Lie algebra generated by a graded vector space $V$ is then the
smallest graded Lie subalgebra $\Li(V) \subseteq T(V)_L$ containing $V$, or equivalently, the intersection of all the graded Lie subalgebras of $T(V)_L$ containing $V$.
We need to prove that if $(V,d)$ is a complex of vector spaces, $\Li(V)$ is a DG-Lie algebra: this can be done by checking that the
differential on $T(V)_L$ restricts to a differential on $\Li(V)$. We will use the following results:

\begin{lemma}Let $\Li(V )_n \subseteq V^{\otimes n}$ be the graded subspace generated
	by all the elements \[[v_1,[v_2,[\dots, [v_{n-1},v_n]\dots]]],\] with $v_1,\dots , v_n \in V$.
	Then $\Li(V)= \bigoplus\limits_{n \geq 1} \Li(V)_n$.
\end{lemma}

\begin{theorem}[Dynkin, Specht, Wever]\label{Thm.DSW}
	Let $V$ be a graded vector space and $H$ a graded Lie algebra over a field $\K$ of characteristic $0$. Let $f\colon V \to H$ be a map of graded vector spaces, and 
	define $F\colon T(V) \to H$ as follows: 
	\[ F(1)=0, \quad F(v)=f(v),\]
	\[F(v_1 \otimes \dots \otimes v_n)= \frac{1}{n}[f(v_1), [f(v_2), \dots [f(v_{n-1}), f(v_n)] \dots ]].\]
	Then the restriction $F\colon \Li(V) \to H$ is the unique morphism of graded Lie algebras extending $f$. 
\end{theorem}

\begin{theorem}[Dynkin projector]\label{cor.dynkin}
For every graded vector space $V$ over a field of characteristic $0$, the map of graded vector spaces \[\rho\colon T(V) \to \Li(V)\quad \rho(V^{\otimes0})=0\]
\[\quad \rho(v_1 \otimes \dots \otimes v_n)= \frac{1}{n}[v_1, [v_2, \dots ,[v_{n-1}, v_n]\dots ]]\] is a projection. Therefore for every $n > 0$
\[\Li(V)_n = \{x \in V^{\otimes n} |\ \rho(x) = x\}.\]
\end{theorem}

For the proofs of the above facts we refer either to \cite[Appendix B]{QuiR} or to \cite[V.4]{Jac}. In the latter, the above facts are proved in the non-graded case, but the graded case is similar. 

The following observation is key:

\begin{remark}\label{rem.Fd=dF} 
	Consider the setup of  Theorem~\ref{Thm.DSW}, where $(V,d)$ is now a complex of vector spaces, $H$ a DG-Lie algebra, and $f\colon V \to H$  a morphism of complexes of vector spaces. Then the map $F\colon T(V)_L \to H$ defined in Theorem~\ref{Thm.DSW} is a morphism of DG-Lie algebras. In fact, 
\begin{align*}
	&F(d(v_1 \otimes \cdots \otimes v_n))= \sum^n_{i=1} (-1)^{\overline{v_1}+ \dots + \overline{v_{i-1}}}F( v_1 \otimes \cdots \otimes dv_i \otimes \cdots \otimes v_n)=\\
	& \frac{1}{n}\sum^n_{i=1}{(-1)^{\overline{v_1}+ \dots + \overline{v_{i-1}}}} [f(v_1),[\dots [f(dv_i), \dots [f(v_{n-1}),f(v_n)] \dots ]]] =  \\
	& \frac{1}{n}\sum^n_{i=1} (-1)^{\overline{f(v_1)}+ \dots + \overline{f(v_{i-1})}}[f(v_1),[\dots[df(v_i), \dots [f(v_{n-1}),f(v_n)] \dots ]]]=\\
	&\frac{1}{n}d([f(v_1),[\dots  [f(v_{n-1}),f(v_n)] \dots ]])= dF(v_1 \otimes \cdots \otimes v_n).
\end{align*}
 The differential on the tensor algebra $(T(V)_L, d)$ restricts then to a differential $d\colon \Li(V) \to \Li(V)$:
in fact, by the above calculation, the differential commutes with the Dynkin projector $\rho$, so that  $d(\Li(V)) \subseteq \Li(V)$.
Therefore $\Li$ is left adjoint to the forgetful functor from DG-Lie algebras to complexes of vector spaces, and it makes sense to call $\Li(V)$ the free DG-Lie algebra generated by the complex of vector spaces $(V,d)$. 
\end{remark}

The functor $\Li$ allows to construct coproducts and pushouts of DG-Lie algebras.

\begin{definition}\label{def.coprod_dgla}
	The coproduct of two DG-Lie algebras $(H, [-,-]_H, d_H)$ and $(M, [-,-]_M, d_M)$ is defined as the quotient $H \amalg M= F/I$ of the free DG-Lie algebra
$F=\Li(H \oplus M)$ by the Lie ideal $I$ generated by the elements of the form $[x,y]_H - [x,y]$ and $[u,v]_M -[u,v]$, with $x,y \in H$, $\ u,v \in M$, where by $[-,-]$ we denote the bracket on $\Li(H \oplus M)$. 

Given two morphisms of DG-Lie algebras $f\colon H \to L$ and $g\colon H \to M$, their pushout is $(L \amalg M) /{J}$, where $J$ is the Lie ideal generated by $f(h) - g(h)$, with $h \in H$.
\end{definition}

The category of DG-Lie algebras over a field $\K$ of characteristic zero is hence complete and cocomplete: coproducts and pushouts are defined above, and products and pullbacks are taken in the category of complexes of vector spaces and endowed with natural brackets.

\begin{remark}\label{remark.sommadiretta}
Note that for a complex of vector spaces $V$ and a DG-Lie algebra $(A, [-,-]_A, d_A)$
there is an isomorphism of DG-Lie algebras
\[A \amalg \Li(V) \cong \Li(A \oplus V)/I_A, \]
where $I_A$ is the Lie ideal generated by $[x,y]_A - [x,y]$, for all $x,y \in A$. In fact, $\Li(A \oplus V)/I_A$ has the universal property of the coproduct: let $N$ be a DG-Lie algebra and $f\colon A \to N$, $g\colon \Li(V) \to N$ maps of DG-Lie algebras.
\begin{center}
	\begin{tikzcd}
		A \arrow[r,  hook] \arrow[rd, "f"'] & \Li(A \oplus V)/I_A \arrow[d, "v", dotted] & \Li(V) \arrow[l,  hook'] \arrow[ld, "g"] \\
		& N       .                               &                                              
	\end{tikzcd}
\end{center}
Since $V \subseteq \Li(V)$, $g$ restricts to a map of complexes $g|_V\colon V \to N$, so we obtain a map of complexes $f+{g}|_V\colon A \oplus V \to N$, which extends to a unique morphism of DG-Lie algebras $v\colon \Li(A \oplus V) \to N$. For $x,y\in A$ one has that 
\begin{align*}
	&v([x,y]_A - [x,y]) = f([x,y]_A) - [v(x),v(y)]_N = f([x,y]_A) -  [f(x),f(y)]_N =0,
\end{align*}
and thus $v(I_A)=0$, so there exists a unique map of DG-Lie algebras $v\colon \Li(A \oplus V)/I_A \to N$.
In particular, note that this implies that $\Li(A)/I_A \cong A$.

Notice also that for a graded vector space $V$ and a graded Lie algebra $A$ there is an isomorphism of graded Lie algebras $A \amalg \Li(V) \cong \Li(A \oplus V)/I_A$, where $A \amalg \Li(V)$ is the coproduct in the category of graded Lie algebras.
\end{remark}

\section{Contractions}\label{sec.contract}

This section contains some technical results about the natural extension of a contraction of complexes of vector spaces to a contraction of tensor algebras and free DG-Lie algebras. These results, which will be needed in Section~\ref{sec.modeldgla}, also allow to prove that the cohomology commutes with the functors $T$ and $\Li$.

Firstly, to fix notation, we recall the definition of contraction of complexes of vector spaces.
\begin{definition}\label{def.contract}
	A contraction is a diagram
	\[ \xymatrix{M\ar@<.4ex>[r]^\iota&N\ar@<.4ex>[l]^\pi\ar@(ul,ur)[]^h} \]
	where $M, N$ are complexes of vector spaces, $h \in \Hom^{-1}_{\K} (N, N )$
	and $\iota, \pi$ are morphisms
	of complexes of vector spaces such that:
	\begin{itemize}
		\item(deformation retraction) $\pi \iota= \Id_M, \ \iota\pi - \Id_N= d_Nh + hd_N$
		\item (annihilation properties) $\pi h = h \iota = h^2=0$
	\end{itemize}
\end{definition}

In particular both $\iota$ and $\pi$ are quasi-isomorphisms.

\begin{lemma}\label{lem.contrazione-T}
	Let $M$ and $N$ be complexes of vector spaces. Every contraction
	\[ \xymatrix{M\ar@<.4ex>[r]^\iota&N\ar@<.4ex>[l]^\pi\ar@(ul,ur)[]^h} \]
	extends canonically to a contraction between tensor algebras
	\begin{equation}\label{diag.tensor_cont}
\xymatrix{T(M)\ar@<.4ex>[r]^{T(\iota)}&T(N).\ar@<.4ex>[l]^{T(\pi)}\ar@(ul,ur)[]^k}
	\end{equation}  
\end{lemma}

\begin{proof}
	Since the map $\iota$ is injective, it is not restrictive to think of $M$ as a subspace of $N$, so that $\iota \colon M \to N$ is an inclusion, and $N= M \oplus W$, where $W$ denotes the kernel of $\pi$.  Since $h\iota=0$ and $\pi h=0$, $h(M)=0$ and the image of $h$ is contained in $W$, so that $h\colon W \to W$ is a contracting homotopy.
	
	Every element of $T(N)$ is linear combination of elements of the form $x_1 \otimes \cdots \otimes x_n$, with $x_i \in M$ or $x_i \in W$. Define:
	\begin{align*}
	 &k(x_1 \otimes \cdots \otimes x_n)=0\quad  \text{if} \ x_i \in M \ \forall i, \\
	&k(x_1 \otimes \cdots \otimes x_n)= \frac{1}{p}\sum^n_{i=1} (-1)^{\overline{x_1} + \dots +\overline{x_{i-1}}} x_1 \otimes \cdots \otimes h(x_i) \otimes \cdots \otimes x_n
	\end{align*}
	if $p$ is the number of $x_i $ in $ W$. We show that \eqref{diag.tensor_cont} is also a contraction: the fact that $k T(\iota)=0$ follows from the definition of $k$, and
	\[T(\pi) k (x_1 \otimes \cdots \otimes x_n) =  \frac{1}{p}\sum^n_{i=1} (-1)^{\overline{x_1} + \dots +\overline{x_{i-1}}} \pi(x_1) \otimes \cdots \otimes \pi h(x_i) \otimes \cdots \otimes \pi(x_n)= 0, \]
	
	\begin{align*}  & k^2(x_1 \otimes \cdots \otimes x_n)= \\
		&= \frac{1}{p^2}\sum_{i, j<i} (-1)^{\overline{x_1} + \dots +\overline{x_{i-1}}+\overline{x_1} + \dots +\overline{x_{j-1}}} x_1 \otimes \cdots \otimes h(x_j) \otimes \cdots h(x_i) \otimes \cdots \otimes x_n  + \\
		& +\frac{1}{p^2}\sum_{i, j>i} (-1)^{\overline{x_1} + \dots +\overline{x_{i-1}}+\overline{x_1} + \dots +\overline{x_{j-1}}-1} x_1 \otimes \cdots \otimes h(x_i) \otimes \cdots h(x_j) \otimes \cdots \otimes x_n=0 .
	\end{align*}
	If $p \neq 0$, let $J \subseteq \{1, \dots, n\}$, $|J|=p$ denote the set of indices such that $i \in J \iff x_i \in W$, then	
	\begin{align*} 
		& (dk +kd)(x_1 \otimes \cdots \otimes x_n)=\\ &= \frac{1}{p}\sum^n_{i=1} (-1)^{\overline{x_1} + \dots +\overline{x_{i-1}}} d( x_1 \otimes \cdots \otimes h(x_i) \otimes \cdots \otimes x_n) + \\
		&+ \sum^n_{i=1} (-1)^{\overline{x_1} + \dots +\overline{x_{i-1}}} k(x_1 \otimes \cdots \otimes dx_i \otimes \cdots \otimes x_n)= \\
		& =\frac{1}{p}\sum^n_{i=1} (x_1 \otimes \cdots \otimes dh(x_i) \otimes \cdots \otimes x_n +  x_1 \otimes \cdots \otimes hd(x_i) \otimes \cdots \otimes x_n)= \\
		& =\frac{1}{p}\sum^n_{i=1}  x_1 \otimes \cdots \otimes (dh(x_i) + hd(x_i)) \otimes \cdots \otimes x_n =\\
		& =\frac{1}{p}\sum_{i\in J}  x_1 \otimes \cdots \otimes (\iota\pi(x_i) -x_i) \otimes \cdots \otimes x_n. 
	\end{align*}
	Finally, since $W=\ker(\pi)$, this is equal to
\begin{align*}
\frac{1}{p}\sum_{i\in J}  x_1 \otimes \cdots \otimes ( -x_i) \otimes \cdots \otimes x_n = - x_1 \otimes \cdots \otimes x_n
 = ( T(\iota) T(\pi)- \Id_N)(x_1 \otimes \cdots \otimes x_n).
\end{align*} 
\end{proof}
\begin{remark}
	A contraction of complexes of vector spaces extends canonically to two different contractions between tensor algebras: for the other one, see \cite[10.1.8]{LMDT}.
\end{remark}
	Let us now derive a formula for the contraction $k$ which will be useful later. 
	In the setup of Lemma~\ref{lem.contrazione-T},  $T(N)= T(M \oplus W) = \bigoplus_{n \geq 0} (M \oplus W)^{\otimes n},$ so defining 
	\begin{equation}\label{eq.Tab}
		T^{a,b}:= \Span \{x_1 \otimes \cdots \otimes x_{a+b}\ |\ a \text{ of the } x_i \ \text{are in } M, b \text{ of the } x_i \ \text{are in } W   \},
	\end{equation} one has that $T(N) = \bigoplus_{a,b \geq 0} T^{a,b}$.
	Then for $x= x_1 \otimes \cdots x_{a+b} \in T^{a,b}$ and $y=  y_1 \otimes \cdots \otimes y_{p,q} \in T^{p,q}$,
	\begin{align*}
		&k(x_1 \otimes \cdots x_{a+b} \otimes y_1 \otimes \cdots \otimes y_{p+q})= \\ &= \frac{1}{b+q}\sum_{i=1}^{a+b} (-1)^{\overline{x_1} + \dots \overline{x_{i-1}} } x_1 \otimes \cdots \otimes h(x_i) \otimes \cdots \otimes x_{a+b} \otimes y_1 \otimes \cdots \otimes y_{p+q} + \\
		&+ \frac{(-1)^{\overline{x}}}{b+q}\sum_{i=1}^{p+q} (-1)^{ \overline{y_1} + \dots + \overline{y_{i-1}} } x_1 \otimes \cdots \otimes x_{a+b} \otimes y_1 \otimes \cdots \otimes h(y_i) \otimes \cdots \otimes y_{p+q}.
	\end{align*}
	Hence for $x \in T^{a,b}$ and $y \in T^{p,q}$,
	\begin{equation}\label{eq.formula_k}
		k(x \otimes y)= \frac{b}{b+q}k(x) \otimes y + (-1)^{\overline{x}}\frac{q}{b+q}x \otimes k(y).
	\end{equation}

\begin{lemma}\label{lem.contrazione-L}
	Let $M$ and $N$ be complexes of vector spaces. Every contraction
	\[ \xymatrix{M\ar@<.4ex>[r]^\iota&N\ar@<.4ex>[l]^\pi\ar@(ul,ur)[]^h}\]
	extends canonically to a contraction between free DG-Lie algebras
	\[ \xymatrix{\Li(M)\ar@<.4ex>[r]^{\Li(\iota)}&\Li(N).\ar@<.4ex>[l]^{\Li(\pi)}\ar@(ul,ur)[]^k} \]
\end{lemma}

\begin{proof}
	Consider again $M$ as a subspace of $N$, so that $N= M \oplus W$, with $W = \ker \pi$. By Lemma~\ref{lem.contrazione-T}  the contraction extends to a contraction of tensor algebras, and it follows from the functoriality of $\Li$ that there exist morphisms of DG-Lie algebras $\Li(\iota) \colon \Li(M) \to \Li(N)$ and $\Li(\pi)\colon \Li(N) \to \Li(M)$ such that $\Li(\pi)\Li(\iota)= \Id_{\Li(M)}$.
	We show that the map $k \colon T(N) \to T(N)$ commutes with the Dynkin projector $\ro$ of Corollary~\ref{cor.dynkin}, so that it induces a contraction of free DG-Lie algebras. Using the formula in~\eqref{eq.formula_k} we obtain for  $x \in T^{a,b} \cap \Li(N)$ and $y \in T^{p,q}\cap \Li(N)$
	\begin{equation}\label{eq.kappaL}
	k([x,y])= \frac{b}{b+q}[k(x), y] + (-1)^{\overline{x}} \frac{q}{b+q}[x, k(y)].
	\end{equation}
	We show that $k\ro(x_1 \otimes \cdots \otimes x_n)= \ro k (x_1 \otimes \cdots \otimes x_n) $ by induction on $n$. For $n=1$, $\ro(x_1)= x_1$ and $k(x_1)= h(x_1)$, so that $k\ro(x_1)= k(x_1) = h(x_1)=\ro h(x_1)= \ro k(x_1)$. 
	By definition of $\ro$, one has that
	\begin{align*}
		\ro (x_1 \otimes \cdots \otimes x_n) =
		 \frac{1}{n}[x_1, [x_2, \dots [x_{n-1}, x_n] \dots ]]&=\frac{n-1}{n}[x_1,[\ro(x_2 \otimes \cdots \otimes x_n)]\\ &=\frac{n-1}{n}[\ro(x_1),[\ro(x_2 \otimes \cdots \otimes x_n)].
	\end{align*} 
	Therefore,  for $x_1 \in T^{a,b}$ and $x_2 \otimes \cdots \otimes x_n \in T^{p,q}$,
	\begin{align*}
		&k \ro(x_1 \otimes \cdots \otimes x_n)= \frac{n-1}{n} k\left([\ro(x_1), \ro(x_2 \otimes \cdots \otimes x_n)]\right) =\\
		&\bigg(\frac{n-1}{n}\bigg) \left( \frac{b}{b+q}[k\ro(x_1), \ro(x_2 \otimes \cdots \otimes x_n)]+ \frac{q}{b+q}(-1)^{\overline{x_1}}[\ro(x_1), k\ro(x_2 \otimes \cdots \otimes x_n)]\right) \\
	\end{align*}
	which by the inductive hypothesis is equal to
	\begin{align*}
		& \left(\frac{n-1}{n}\right) \left( \frac{b}{b+q}[\ro k(x_1), \ro(x_2 \otimes \cdots \otimes x_n)]+  \frac{q}{b+q}(-1)^{\overline{x_1}}[\ro(x_1), \ro k(x_2 \otimes \cdots \otimes x_n)] \right).
	\end{align*}
	On the other hand,
	\begin{align*}
		&\ro k(x_1 \otimes \cdots \otimes x_n)= \\
		&\frac{b}{b+q} \ro (k(x_1)\otimes x_2 \otimes \cdots \otimes x_n) + (-1)^{\overline{x_1}} \frac{q}{b+q}\ro(x_1 \otimes k(x_2 \otimes \cdots \otimes x_n))=\\
		&\left(\frac{n-1}{n} \right) \left( \frac{b}{b+q} [\ro k(x_1), \ro (x_2 \otimes \cdots \otimes x_n)] + (-1)^{\overline{x_1}}  \frac{q}{b+q} [\ro (x_1), \ro k(x_2 \otimes \cdots \otimes x_n)]\right),
	\end{align*}
which concludes the proof. 
\end{proof}

As a consequence of 
the previous lemmas, one can prove the commutativity of the cohomology with the functors $T$ and $\Li$, see also \cite[Appendix B]{QuiR}.

\begin{proposition}\label{prop.isoH}
	For any complex of vector spaces $(V,d)$ there are isomorphisms
	\begin{enumerate}
		\item $H^*(T(V)) \cong T(H^*(V))$,
		\item $H^*(\Li(V)) \cong \Li(H^*(V))$.
	\end{enumerate}
\end{proposition}

\begin{proof}
There exists a splitting $V=H\oplus W$, with $W$ an acyclic complex and $H\cong H^*(V)$ a complex with trivial differential. The complex $W$ is contractible, so that there exists $\gamma \in \Hom^{-1}_\K(W,W)$ with $d\gamma+\gamma d=-\Id_W$. 
 Since $\gamma^2$ is not necessarily zero (as it should be according to Definition~\ref{def.contract}), we define
$h :=  \gamma d \gamma \in  \Hom^{-1}_{\K}(V,V)$. One can check that $dh + hd = -\Id_V$ still holds, and that additionally $h^2=0$.
This map can be extended to $h\colon H \oplus W \to H \oplus W$ by setting $h(H)=0$, so that the diagram
	\[ \xymatrix{H\ar@<.4ex>[r]^\iota&{V}\ar@<.4ex>[l]^\pi\ar@(ul,ur)[]^h} \]
	is a contraction. By Lemma~\ref{lem.contrazione-T}, this extends to a contraction of tensor algebras
	\[ \xymatrix{T(H)\ar@<.4ex>[r]^\iota&{T(V).}\ar@<.4ex>[l]^\pi\ar@(ul,ur)[]^k} \]
	This implies that $H^*(T(H)) \cong H^*(T(V))$, and since $H$ has trivial differential, $H^*(T(H)) = T(H) \cong T(H^*(V))$.

	Let $\Li(V)$ be the free DG-Lie algebra generated by $(V,d)$. By Lemma~\ref{lem.contrazione-L} the diagram 
	\[ \xymatrix{\Li(H)\ar@<.4ex>[r]^\iota&{\Li(V)}\ar@<.4ex>[l]^\pi\ar@(ul,ur)[]^k} \]
	is a contraction, which implies $H^*(\Li(H)) \cong H^*(\Li(V))$, and since  $H$ has trivial differential, $H^*(\Li(H))= \Li(H) \cong \Li(H^*(V))$.
\end{proof}

\section{Model and pre-model structures}\label{sec.model}

The topics of this section are a brief reminder about the definition and basic properties of model categories, the definition of left pre-model structures, and the fact that a left pre-model structure naturally induces a model structure. 
The notion of left pre-model structure will be used in the next section to give a proof of the existence of the model structure on the category of unbounded DG-Lie algebras, which is the goal of the paper.

In a model category, two out of the three classes of weak equivalences, cofibrations and fibrations determine the third. However it is important to note that, for example, if
$\sW$ and $\sF$ are two classes satisfying (M1) and (M2) in Definition~\ref{def.mod} below, in general they do not extend to a model structure.  It is typical to try to construct a model structure by establishing two of the three classes and seeing whether they induce a model structure. 
Left pre-model structures are a useful instrument to prove the existence of a model structure when one has fixed the weak equivalences, the fibrations, and a subclass of the cofibrations (see also \cite[Section 3.2]{tesiGreco} and \cite{defodiag}).

Firstly, we briefly recall the definitions of lifting properties, retracts and that of a model structure.
The standard references for model categories, introduced by Quillen in \cite{QuiH}, are the books \cite{Hir,Hov}.

\begin{definition} Consider two morphisms $i\colon A\to B$ and $p\colon C\to D$ in a category $\bM$.
 Then $i$ has the left lifting property with respect to $p$ or, equivalently, $p$ has the right lifting
	property with respect to $i$ if for every commutative diagram of the form
	\begin{center}
		\begin{tikzcd}
			A \arrow[r, "f"] \arrow[d, "i"'] & C \arrow[d, "p"] \\
			B \arrow[r, "g"'] \arrow[ru, dotted, "h"] & D
		\end{tikzcd}
	\end{center}  
	there exists a lift $h \colon B \to Y$ such that $hi = f$  and $ph = g$.
\end{definition}

\begin{definition}
	A morphism  $f$ is called a retract of a morphism $g$ if there exists a commutative diagram of the form
	\begin{center}
		\begin{tikzcd}
			A \arrow[r] \arrow[d, "f"'] \arrow[rr, bend left, "Id_A"] & B  \arrow[r] \arrow[d, "g"] & A \arrow[d, "f"] \\
			C  \arrow[r] \arrow[rr, bend right, "Id_C"]& D  \arrow[r] & C .
		\end{tikzcd}
	\end{center}
\end{definition}
\begin{definition}\label{def.mod}
	A model structure on a complete and cocomplete category $\bM$  is the data of three classes of maps, called weak equivalences, 
	fibrations and cofibrations, which satisfy the following four axioms:
	
	\begin{enumerate}
		\item[(M1)] (2-out-of-3) If $f$ and $g$ are morphisms in $\bM$ such that the composition $gf$ is defined, and two out of the three $f$, $g$ and $gf$ are weak equivalences, so is the third.
		\item[(M2)] (Retracts) If $f$ and $g$ are maps in $\bM$ such that $f$ is a retract of $g$, and $g$ is a weak equivalence, a cofibration or a fibration, then so is $f$.
		\item[(M3)] (Lifting) A trivial fibration is map which is both a fibration and a weak equivalence; a trivial cofibration is map which is both a cofibration and a weak equivalence.
		\begin{enumerate}
			\item Cofibrations have the left lifting property with respect to trivial fibrations.
			\item Trivial cofibrations have the left lifting property with respect to fibrations.
		\end{enumerate}
		\item[(M4)](Factorisation) Every morphism  $g$ in $\bM$ admits  two factorisations: 
		\begin{enumerate}
			\item 
			$g=qj$, where $j$ is a trivial cofibration and  $q$ is a fibration,
			\item 
			$g=pi$, where $i$ is a cofibration and $p$ is a trivial fibration.
		\end{enumerate}
	\end{enumerate}
\end{definition}

As usual, the classes of weak equivalences, fibrations and cofibrations will be denoted respectively by $\sW,\sF$ and $\sC$. We shall denote by $\sF\sW=\sF\cap \sW$ the class of trivial fibrations and by  $\sC\sW=\sC\cap \sW$ the class of trivial cofibrations. An object is said to be cofibrant if the initial map is a cofibration; it is called fibrant if the terminal map is a fibration. 


Since they will be used later on, a few standard results in model category theory are recalled for the reader's convenience. 
We refer to \cite[Section 7.2]{Hir} or \cite[Section 1.1]{Hov} for the proofs.

\begin{lemma}\label{ret-sol}
	If $j$ has the left (right) lifting property with respect to $f$, and $i$ is a retract of $j$, then $i$ has the left (right) lifting property with respect to $f$. 
\end{lemma}

\begin{proposition}[Retract Argument]\label{retractargument}
	Let $g$ be a map which can be factored as $g=pi$.
	\begin{enumerate}
			\item If $g$ has the left lifting property with respect to $p$
			then $g$ is a retract of $i$.
			\item If $g$ has the right lifting property with respect to $i$
			then $g$ is a retract of $p$.
		\end{enumerate}
\end{proposition}

\begin{lemma}\label{lem.liftingdefinition} 
	Let $\bM$ be a model category,
	\begin{enumerate}
		\item a map in $\bM$ that has the left lifting property with respect to all trivial fibrations is a cofibration;
		\item a map in $\bM$ that has the right lifting property with respect to all trivial cofibrations is a fibration.
	\end{enumerate}
\end{lemma}

\begin{lemma}\label{lem.comp_lift}
Let $f$ and $g$ be maps such that the composition $fg$ is defined. If $f$ and $g$ have the left (right) lifting property with respect to a map $p$, then the composition $f g$ has the left (right) lifting property with respect to $p$.
\end{lemma}

Next, we recall the definition of left pre-model structure, and prove that a left pre-model structure extends naturally to a model structure. 

\begin{definition}\label{def.leftpremodel}
	A left pre-model structure on $\bM$ is the data of four classes of maps $\mathcal{W},\mathcal{F}, \mathcal{C'},\mathcal{CW'}$ such that:  
	\begin{enumerate}
		\item[(L1)] The maps in $\mathcal{W}$ satisfy the 2-out-of-3 property;
		\item[(L2)] The classes $\mathcal{W}$ and $\mathcal{F}$ are closed under retracts;
		\item[(L3)] There is an inclusion $\mathcal{CW'} \subset \mathcal{W}$;
		\item[(L4)] The maps in $\mathcal{C'}$ have the left lifting property with respect to the maps in $\mathcal{F}\cap \mathcal{W}$; the maps in $\mathcal{CW'}$ have the left lifting property with respect to the maps in $\mathcal{F}$;
		\item[(L5)]  Every map $g$ in $\bM$ has two  factorisations: 
		\begin{enumerate}
			\item $g=qj$, where $i$ is in $\mathcal{CW'}$ and $p$ is in $\mathcal{F}$;
			\item $g=pi$, where $j$ is in $\mathcal{C'}$ and  $q$ is in $\mathcal{F}\cap \mathcal{W}$.
		\end{enumerate}
	\end{enumerate}
\end{definition}

There exists a dual definition of right pre-model structure, useful when one has fixed the classes of weak equivalences, cofibrations, and some particularly easy fibrations. 
A left pre-model structure gives rise to a unique model structure in the following way:
\begin{theorem}\label{thm.premodel}
	Given a left pre-model structure $\mathcal{W},\mathcal{F}, \mathcal{C'},\mathcal{CW'} $  there exists a unique model structure  where the weak equivalences are the maps in $\mathcal{W}$ and the fibrations are the maps in $\mathcal{F}$. Notably, the cofibrations are the retracts of $\mathcal{C'}$, and the trivial cofibrations are the retracts of $\mathcal{CW'}$.
\end{theorem}

We refer to \cite{defodiag} for the proof. 

\section{Model structure on DG-Lie algebras}\label{sec.modeldgla}

In this section we work in the category of unbounded DG-Lie algebras over a fixed field $\K$ of characteristic zero.
There is a model structure on this category, constructed in a more general context in \cite{Hin}, where weak equivalences are quasi-isomorphisms and fibrations are surjective morphisms. Cofibrations are then by Lemma~\ref{lem.liftingdefinition} the maps with the left lifting property with respect to surjective quasi-isomorphisms.
The goal of this section is to give an elementary proof of the existence of this model structure, which will be done via the notions of pre-model structure of the previous section, and of semifree extension of DG-Lie algebras.
 We also give a more explicit description of cofibrations as retracts of semifree extensions.

\begin{lemma}\label{lem.derivazamalg}
	Let $A$ be a graded Lie algebra and $V$ a graded vector space.
	There is a bijection between the degree $n$ derivations of $A \amalg \Li(V)$ that preserve the Lie subalgebra $A$ and pairs $(f,g)$ such that $f \in \Der_\K^n(A,A)$ and $g \in \Hom_\K^n(V, A \amalg \Li(V))$.
\end{lemma}

\begin{proof}
	Given $\phi \in \Der_\K^n (A \amalg \Li(V))$ such that $\phi(A) \subseteq A$, then $\phi|_A $  is a degree $n$ derivation of $A$ and $\phi|_V \colon V \to A \amalg \Li(V)$ is a linear map of degree $n$. Conversely, given $g \in \Hom_\K^n(V, A \amalg \Li(V))$ and  $f \in \Der_\K^n(A,A)$, we can define
	$\phi(v)= g(v)$ for all $v \in V$, $\phi(a)=f(a)$ for all $a \in A$, and
	\[\phi([x_1, \dots [x_{n-1}, x_n] \dots ])= \sum_{i=1}^m (-1)^{n(\overline{x_1} + \dots + \overline{x_{i-1}})} [x_1, \dots [\phi(x_i), \dots [x_{n-1}, x_n] \dots ]]. \]
	Notice that a map of this form is not in general well defined, but in this case $\phi$ is because $f$ is already a derivation on $A$.
	
\end{proof}

The lemma implies that if $A$ is a DG-Lie algebra and $V$ a graded vector space, a differential on $A \amalg \Li(V)$ is determined by the differential on $A$ and by a linear map $g \colon V \to A \amalg \Li(V)$ of degree $1$.

\begin{definition}[Free extension]
	A \textit{free extension} in the category of DG-Lie algebras is an inclusion $A \to  A \amalg \Li(V)$, with $(V,d)$ an acyclic complex of vector spaces, and $A \amalg \Li(V)$ the coproduct of DG-Lie algebras. 
\end{definition}

\begin{definition}[Semifree extension]
	
	An \textit{elementary semifree extension} in the category of DG-Lie algebras is a  morphism of DG-Lie algebras $f \colon A \to B$ such that $B$ is isomorphic as a graded Lie algebra  to the coproduct of graded Lie algebras $A \amalg \Li(V)$, with $V$ a graded vector space, in such a way that $f$ is identified with the inclusion $A \to A \amalg \Li(V)$. Moreover, denoting by $d$ the differential induced on $A\amalg \Li(V)$ by said isomorphism, we have $dV \subseteq A$.

	A \textit{semifree extension} is the countable composition of elementary semifree extensions. 
\end{definition}

Every free extension is a semifree extension: let $(V,d)$ be an acyclic complex of vector spaces, let $Z^*(V) \subset V$ be the graded vector spaces of cocycles, and let $V=Z^*(V)\oplus W$ be a splitting of graded vector spaces. Notice that since $V$ is acyclic, $Z^*(V)=B^*(V)= d(V)= d(W)$. The free extension $A \to A \amalg \Li(V)$ is given by the composition of two elementary semifree extensions 
\begin{center}
	\begin{tikzcd}
		A \arrow[r] & A\amalg \Li(Z^*(V)) \arrow[r] & A\amalg \Li(V),
	\end{tikzcd}
\end{center}
where the second map is the inclusion  \begin{center}
	\begin{tikzcd}
		A \amalg \Li(Z^*(V)) \arrow[r] & (A\amalg \Li(Z^*(V)) )\amalg \Li(W) = A \amalg \Li(V)
	\end{tikzcd}
\end{center}
with $d(W) = Z^*(V) \subseteq A\amalg \Li(Z^*(V))$, hence it is an elementary semifree extension.

\begin{theorem}\label{thm.model}
	There is a 
	model structure on the category of unbounded DG-Lie algebras over a field of characteristic zero, where weak equivalences are quasi-isomorphisms, fibrations are surjective morphisms,  cofibrations are retracts of semifree extensions, and trivial cofibrations are  retracts of free extensions.
\end{theorem}

The theorem is proved using the notion of left pre-model structure of Section~\ref{sec.model};
in particular, we set:
\begin{enumerate}
	\item $\sW$ the quasi-isomorphisms,
	\item $\sF$ the surjective maps,
	\item $\sC'$ the semifree extensions,
	\item $\sC\sW'$ the free extensions.
\end{enumerate}

We now show that $\sW, \sF, \sC', \sC\sW'$ satisfy the axioms of Definition~\ref{def.leftpremodel}; the proof is split into the following five 
propositions.
It is clear that quasi-isomorphisms have the 2-out-of-3 property, and that the classes $\sW$ and $\sF$ are closed by retracts, so we begin by proving that $\sC\sW' \subseteq \sW$.

\begin{proposition}
	Free extensions are quasi-isomorphisms.
\end{proposition}

\begin{proof}
	Let $A$ be a DG-Lie algebra, $(V,d)$ an acyclic complex of vector spaces, and let $f \colon A \to A \amalg \Li(V)$ be a free extension. 
		If $A=0$, the fact that the free DG-Lie algebra $\Li(V)$ is quasi-isomorphic to $0$ follows from Proposition~\ref{prop.isoH}:  $H^*(\Li(V)) \cong \Li (H^*(V))=0$ since $V$ is acyclic.
		
		Let now $A$ be different from zero;
	by Remark~\ref{remark.sommadiretta}, $A \amalg \Li(V) \cong \Li(A \oplus V)/I_A$, where $I_A$ is the Lie ideal generated by $[a,b]_A -[a,b]$, for all $a,b \in A$. Since $(V,d)$ is an acyclic complex of vector spaces, the identity $Id_V$ is a coboundary in $\Hom^*_{\K}(V,V)$, i.e. there exists $\gamma \in \Hom^{-1}_{\K}(V,V) $ such that $d\gamma+\gamma d=-\Id_V$. 
	We wish to construct a contraction $0\; \rightleftarrows \; V $ as in Definition~\ref{def.contract}; since $\gamma^2$ is not necessarily zero, we define
	$h :=  \gamma d \gamma \in  \Hom^{-1}_{\K}(V,V)$. One can check that $dh + hd = -\Id_V$ still holds, and that additionally $h^2=0$.
	
	We can now extend $h$ to a map $h\colon A \oplus V \to A \oplus V$ by setting $h(A)=0$. Let $p_1 \colon A \oplus V \to A $ be the  projection and $i_1 \colon A \to A\oplus V$ the injection, so that $p_1i_1=\Id_A$; then the diagram
	\[ \xymatrix{A\ar@<.4ex>[r]^-{i_1}&{A \oplus V}\ar@<.4ex>[l]^-{p_1}\ar@(ul,ur)[]^{h}} \]
	is a contraction.  By Lemma~\ref{lem.contrazione-L}, it extends to a contraction 
	\[\xymatrix{\Li(A)\ar@<.4ex>[r]^-{i_1}&{\Li(A \oplus V})\ar@<.4ex>[l]^-{p_1}\ar@(ul,ur)[]^{k}}, \]
	where for brevity we have denoted $\Li(i_1)$ by $i_1$ and $\Li(p_1)$ by $p_1$, and where the map $k$ is such that $k(A)=0$ and
	\begin{equation}\label{eq.kappa2}
		 k([x,y])=  \frac{b}{b+q}[k(x), y] + (-1)^{\overline{x}} \frac{q}{b+q}[x, k(y)],
	\end{equation}
	for $x \in L^{a,b}$ and $y \in L^{p,q}$, see \eqref{eq.kappaL}, where
	 \[L^{a,b} = \Span \{[x_1, [\dots[x_{a+b-1}, x_{a+b}] \dots]]\ |\ a \text{ of the } x_i \ \text{are in } A, b \text{ of the } x_i \ \text{are in } V  \}.\]
	
	We now show that there are induced maps $\tilde{i_1}, \tilde{p_1}$ and $\tilde{k}$ such that the diagram
	\[ \xymatrix{\Li(A)/I_A\ar@<.4ex>[r]^-{\tilde{i_1}}&\Li(A \oplus V)/I_A\ar@<.4ex>[l]^-{\tilde{p_1}}\ar@(ul,ur)[]^{\tilde{k}}} \]
	is a contraction. The map $\tilde{i_1}$ is induced by $i_1$, because it preserves the ideal $I_A$: 
	 for $x,y \in A$ one has $i_1([x,y]_A - [x,y])= i_1([x,y]_A) - [i_1(x), i_1(y)] = [x,y]_A - [x, y] \in I_A$; likewise  $\tilde{p_1}$ is induced by $p_1$.
	The Lie ideal $I_A$ is generated by the vector subspace \[U=\Span \{ [a,b]_A -[a,b]\ |\ a, b \in A \};\]
	we show first that $k$ preserves $U$. 
	In fact, since $a,b$ and $[a,b]_A$ are in $A$, $k(a)=k(b)=k([a,b]_A)=0$,
	and by \eqref{eq.kappa2}, there exist $\alpha, \beta \in \mathbb{Q}$ such that
	\begin{equation}\label{eq.kappa}
		k ([a,b]_A -[a,b])= - k([a,b])= - \alpha [k(a), b] - (-1)^{\overline{a}} \beta[a, k(b)]=0.
	\end{equation}
 An element of the ideal $I_A$ is a linear combination of elements of the form
	\[ [x_1, [x_2, \ldots [x_n , u] \ldots ]], \quad n \in \N, \ \ u \in U;\]
	we prove that  	$k(I_A) \subseteq I_A$ by induction on $n$. The case $n=0$ follows from the fact that $k$ preserves $U$. 
	Assume that $k ([x_1, [x_2, \ldots [x_{n} , u] \ldots ]]) $ belongs to $I_A$, then by Equation~\eqref{eq.kappaL} there exist some  $\alpha', \beta' \in \mathbb{Q}$  such that
	\[ k ([x_0, [x_1, \ldots [x_{n} , u] \ldots ]]) = \alpha'  [k(x_0), [x_1, \ldots [x_{n} , u] \ldots ]] + \beta' [x_0, k([x_1, \ldots [x_{n} , u] \ldots ])]. \]
	Note that $ [x_1, \ldots [x_{n} , u] \ldots ]]$ belongs to the ideal $I_A$, and hence so does $\  [k(x_0), [x_1, \ldots [x_{n} , u] \ldots ]]$.
	  By the inductive hypothesis,  $k([x_1, \ldots [x_{n} , u] \ldots ])$ belongs to the ideal $I_A$, and therefore also  {$[x_0, k([x_1, \ldots [x_{n} , u] \ldots ])]$} belongs to $I_A$.

	Finally, since $\Li(A)/I_A \cong A$ and $\Li(A \oplus V)/I_A \cong A \amalg \Li(V)$ in such a way that $f$ is identified with the inclusion $\tilde{i_1}$, we have a contraction 
	\[ \xymatrix{A\ar@<.4ex>[r]^-{f}& {A \amalg \Li(V)},\ar@<.4ex>[l]^{}\ar@(ul,ur)[]^k}\]
	therefore $f\colon A \to A \amalg \Li(V)$ is a quasi-isomorphism.
\end{proof}

\begin{proposition}\label{prop.lifting1}
	For every commutative diagram of DG-Lie algebras
	\begin{center}
		\begin{tikzcd}
			A \arrow[d, "i"'] \arrow[r ]                        & C \arrow[d, "g"] \\
		B \arrow[r] \arrow[ru, "h", dotted] & D,               
		\end{tikzcd}
	\end{center}
	where $i$ is a semifree extension and $g$ is a surjective quasi-isomorphism, there exists a lifting $h\colon A \amalg \Li(V) \to C$ making both triangles commute.
\end{proposition}

\begin{proof}
	By Lemma~\ref{lem.comp_lift} we can suppose without any loss of generality that $i$ is an elementary semifree extension,  $i \colon A \to A \amalg \Li(V)$, with  $d(V) \subseteq A$, and consider a diagram of the form
		\begin{center}
		\begin{tikzcd}
			A \arrow[d, "i"'] \arrow[r, "\gamma"]                        & C \arrow[d, "g"] \\
			 A \amalg \Li(V) \arrow[r, "\beta"'] \arrow[ru, "h", dotted] & D.           
		\end{tikzcd}
	\end{center}
	To define a lifting $h\colon A \amalg \Li(V) \to C$ we need only define a graded linear map ${k}\colon V \to C$ such that the map $h$ induced by universal properties of the free graded Lie algebra and of the coproduct of graded Lie algebras has the required properties. 
	\begin{center}
		\begin{tikzcd}
			A \arrow[r, "i"] \arrow[rd, "\gamma"', bend right=20] & A \amalg \Li(V) \arrow[d, "\exists ! h",] & \Li(V) \arrow[l] \arrow[ld, "k", bend left=20] \\
			& C                                       &                                            
		\end{tikzcd}
	\end{center}
	Let $\{v_i\} $ be generators of the graded vector space $V$. By definition of elementary semifree extension the elements $dv_i$ belong to $A$, so the only possible definition for the lifting is $h(dv_i)= \gamma(dv_i)$. Then $dh(dv_i)=0$, and $gh(dv_i)=g\gamma(dv_i)=\beta(dv_i)= d\beta(v_i)$. Since $g$ is a quasi-isomorphism the $h(dv_i)$ are exact in $C$, so there exist $c_i \in C$ such that $h(dv_i)=dc_i$. A surjective quasi-isomorphism is surjective on cocycles: for any $x \in D$ such that $dx=0$ there exists $y \in C$ such that $dy=0$ and $g([y]) =[x]$, so that $ g(y)= x + dz,$ with $z \in D$. By the surjectivity, $z= g(t)$, so that $x= g(y) - dg(t) = g(y-dt)$.
	Therefore, since \[d (g(c_i) -\beta(v_i))= g(dc_i)-\beta(dv_i)=gh(dv_i)-gh(dv_i)=0\] there exist $c_i' \in Z^*(C)$ such that $\beta(v_i)= g(c_i + c_i')$. We then set $h(v_i):= c_i + c_i'$, which is the required lifting, as $gh(v_i)= \beta(v_i)$ and $dh(v_i)= dc_i + dc_i' =dc_i= h(dv_i)$. 
\end{proof}

\begin{proposition}
	For every commutative diagram of DG-Lie algebras
	\begin{center}
		\begin{tikzcd}
			A \arrow[d, "i"'] \arrow[r, "\gamma"]                        & C \arrow[d, "g"] \\
			A \amalg \Li(V) \arrow[r, "\beta"'] \arrow[ru, "h", dotted] & D,               
		\end{tikzcd}
	\end{center}
	where $i$ is a free extension and $g$ is surjective, there exists a lifting $h\colon A \amalg \Li(V) \to C$ making both triangles commute.
\end{proposition}

\begin{proof}
 It is enough to define a map of complexes $h \colon V \to C$: in fact, by the universal property of the free DG-Lie algebra this extends to a DG-Lie morphism $h \colon \Li(V) \to C$, and then to a DG-Lie morphism  $h \colon A \amalg \Li(V) \to C$ by the universal property of the coproduct. 
 
 	Since $V$ is acyclic, there is a splitting of graded vector spaces $V= Z^*(V) \oplus W$ with $d(V)=d(W)=B^*(V)=Z^*(V)$.
	We start by defining a graded linear map $W \to C$. Taking generators $w_i \in W$ we set $h(w_i):= c_i$, where $c_i \in C$ are elements such that $g(c_i)=\beta(w_i)$, which exist because of the surjectivity of $g$. As remarked before, every $z \in Z^*(V)$ is a coboundary, so $z=dx$ with $x \in W$: we then set $h(z)=h(dx):= dh(x)$, so that $h$ commutes with differentials. The induced map $h \colon A \amalg \Li(V) \to C$ is such that $hi=\gamma $ and $gh=\beta$.
\end{proof}

\begin{proposition}
	Every morphism $f\colon A \to B$ of DG-Lie algebras can be factored as a free extension followed by a surjective morphism.
\end{proposition}

\begin{center}
	\begin{tikzcd}
		& A\amalg \Li(V) \arrow[rd, "g", bend left=20] &   \\
		A \arrow[rr, "f"] \arrow[ru, "i", bend left=20] &                                            & B
	\end{tikzcd}
\end{center}

\begin{proof}
	We construct an acyclic complex of vector spaces $(V,d)$ such that there is a surjective map of complexes $V \to B$, which induces a surjective morphism of DG-Lie algebras $\Li(V) \to B$,
 and hence a surjective morphism of DG-Lie algebras
	$g\colon A\amalg \Li(V) \to B$ such that $gi=f$. 

	Let $\{b^i_j\}$ be elements of $B^i$ whose cohomology classes generate $H^i(B)$. For every $b^i_j$, let $c^i_j$ be a new element of degree $i-1$, and set $V^{i-1}:= B^{i-1} \oplus \K \langle c^i_j \rangle$,  $dc^i_j=b^i_j$. There is a surjection $V \to B$, and $V$ is by construction acyclic.
\end{proof}

\begin{proposition}
	Every morphism $f\colon L \to M$ of DG-Lie algebras can be factored as a semifree extension followed by a surjective quasi-isomorphism.
\end{proposition}

\begin{proof}
	Using the previous proposition, factor $f$ as
	\begin{center}
		\begin{tikzcd}
			L \arrow[r, "j"] & C \arrow[r, "g", two heads] & M,
		\end{tikzcd}
	\end{center}
	with $j$ a free extension and $g$ a surjective map. Since a free extension is in particular a semifree extension, and semifree extensions are closed by composition by definition, it is sufficient to factor the surjective map $g\colon C \to M$ as a semifree extension followed by a quasi-isomorphism.
	We construct a sequence of DG-Lie algebras 
	and a sequence of morphisms $f_n \colon C_n \to M$
	\begin{center}
		\begin{tikzcd}
			C=C_0 \arrow[d, "g=f_0"', two heads] \arrow[r, "i_o", hook] & C_1 \arrow[r, "i_1", hook] \arrow[ld, "f_1", near start, bend left=10] & C_2 \arrow[r, hook] \arrow[lld, "f_2", near start, bend left=15] & \cdots \arrow[r, "i_{n-1}", hook] & C_n \arrow[r, "i_n"] \arrow[lllld, "f_n",near start, bend left=15] & \cdots \\
			M & & & & &       
		\end{tikzcd}
	\end{center} such that:
	\begin{enumerate}
		\item $C_n = C_{n-1} \amalg \Li(V_n)$ as a coproduct of graded Lie algebras, with $d(V_n) \subseteq C_{n-1}$, \\ $i_{n-1} \colon C_{n-1} \to C_n$ is the natural inclusion in the coproduct;
		\item $f_{n}$ extends $f_{n-1}$;
		\item $f_1 \colon Z^{*}(C_1) \to Z^*(M)$ is surjective;
		\item $f_n^{-1}(B^*(M)) \cap Z^* (C_n) \subseteq B^*(C_{n+1}) \cap C_n, \ \forall n >0$
	\end{enumerate}
	Setting $\tilde{C}= \colim_n C_n$, $\tilde{f}= \colim_n f_n$ and $i \colon C \to \tilde{C}$ the natural inclusion, we have that 
	\begin{center}
		\begin{tikzcd}
			C \arrow[r, "i"] & \tilde{C} \arrow[r, "\tilde{f}"] & M
		\end{tikzcd}
	\end{center}
	is a factorisation with the required properties. Note that by the surjectivity of $f_0=g$ all the maps $f_n$ are surjective, and thus so is $\tilde{f}$.
	
	We begin by setting $C_0=C$, $f_0=g$, and $V_1= Z^{*}(M)$, $d(V_1)=0$. The inclusion $h_1 \colon V_1 \to M$ induces the DG-Lie algebra map $f_1 \colon C_1:= C \amalg \Li(V_1) \to M$, which is surjective on cocycles by construction. For the next step, let $w_j$ be generators of $f_1^{-1}(B^*(M)) \cap Z^*(C_1)$, so that $dw_j=0$ and $f_1(w_j)= dn_j$. Set $V_2 =\K\langle y_j\rangle $, $dy_j= w_j \in C_1$ and $h_2 \colon V_2 \to M$, $h_2(y_j)= n_j$, then $h_2$ induces the DG-Lie algebra map $f_2\colon C_2= C_1 \amalg \Li(V_2) \to M$ with the required properties.
	
	Assume now that we have constructed $C_{n}$ and $f_{n}\colon C_{n} \to M$, and let $v_i$ be a set of generators of $f_n^{-1}(B^*(M)) \cap Z^* (C_n)$, so that $dv_i=0$ and $f_n(v_i)=dm_i$. Define $V_{n+1}= \K \langle x_i \rangle$, $dx_i= v_i \in C_n$, and $h_{n+1} \colon V_{n+1} \to M$ as $h_{n+1}(x_i)= m_i$. Define $C_{n+1} = C_{n} \amalg \Li(V_{n+1})$, and $f_{n+1}$ as the natural DG-Lie algebra map induced by $h_{n+1}$.
\end{proof}

%

\begin{remark}
	Consider the category of bounded above DG-Lie algebras  $\mathbf{DGLA}^{< 0 }_{\mathbb{K}}$, and the faithful inclusion functor
	\[ \iota \colon  \mathbf{DGLA}^{<0}_{\mathbb{K}} \to  \mathbf{DGLA}_{\mathbb{K}}.\]
	The canonical truncation of an unbounded DG-Lie algebra $L, $defined as
	\[ \tau (L)^i = \begin{cases}
		L^i\quad &\text{if}\ \   i < -1,\\
		\ker(d^r) \quad &\text{if}\ \  i=-1,\\
		0\quad &\text{if}\ \ i >-1
	\end{cases}\]
gives a a functor $\tau \colon  \mathbf{DGLA}_{\mathbb{K}} \to \mathbf{DGLA}^{< 0 }_{\mathbb{K}}$ which is right adjoint to $\iota$.
We show here that 
{by applying $\tau$ one obtains exactly the model structure on  $\mathbf{DGLA}^{< 0 }_{\mathbb{K}}$ considered in \cite{QuiR}}. 
In this model structure, the fibrations -- which we will denote by  $\sF^{<0}$ -- are the maps surjective in all degrees $i <-1$ and the weak equivalences -- which we will denote by $\sW^{<0}$ -- are the quasi-isomorphisms. Cofibrations, which we denote by $\sC^{<0}$, are defined in accordance to Lemma~\ref{lem.liftingdefinition} as the maps with the left lifting property with respect to trivial fibrations. 

In fact, $\tau$ applied to a fibration in $\mathbf{DGLA}_{\mathbb{K}}$, i.e., to a surjective map, gives exactly a map which is surjective in all degrees $ i < -1$. 
Denote as usual by $\sF, \sC$ and $\sW$ the fibrations, cofibrations and weak equivalences of the model structure on $\mathbf{DGLA}$. We then have that 
\[ \iota ( \sF^{<0} \cap \sW^{<0}) \subseteq \sF \cap \sW:\]
 any quasi-isomorphism $f \colon L \to M$ in $\mathbf{DGLA}^{<0}_{\mathbb{K}}$ which is surjective in degrees $i < -1$ is also surjective in degree $-1$. In fact, for any $x \in M^{-1}$ one has $dx=0$, so by the fact that $f$ is a quasi-isomorphism there exists $y \in L^{-1}$ such that $f(y)= x + dz$, with $z \in M^{-2}$. Since $f$ is surjective in degree $-2$, there exists $t \in L^{-2}$ such that $f(t)=z$, hence $x= f(y)-d f(t)= f (y-dt)$.  By the fact that $\tau \circ \iota$ is the identity one obtains the inclusion $ \sF^{< 0} \cap \sW^{<0} \subseteq \tau(\sF \cap \sW)$. The other inclusion  $\tau (\sF \cap \sW) \subseteq \sF^{<0} \cap \sW^{<0}$ is clear, so that $	\tau (\sF \cap \sW) = \sF^{<0} \cap \sW^{<0}$ . 

The maps in $\sC^{<0}$ are by definition those with the left lifting property with regards to the maps in $\sF^{<0} \cap \sW^{<0} = \tau (\sF \cap \sW)$. By the adjointness of $\tau$ and $\iota$, this means that the maps in $\iota (\sC^{<0})$ are exactly those with the left lifting property with regards to the maps in $\sF \cap \sW$, so again by Lemma~\ref{lem.liftingdefinition} they coincide with the cofibrations in $\mathbf{DGLA}$. Applying $\tau$ to the identity $\iota (\sC^{<0 })= \sC$, we conclude that $\sC^{< 0} = \tau(\sC)$, so that the functors $\iota$ and $\tau$ preserve all the model structure. 
\end{remark}

\section{Examples}\label{sec.ex}

This section contains an example of a DG-Lie algebra which is not cofibrant, in spite of the underlying graded Lie algebra being free, and the proof that the cobar construction of a locally conilpotent cocommutative coalgebra is a semifree DG-Lie algebra.

In the first part of the section we assume that the reader has a basic knowledge of deformation functors associated to differential graded Lie algebras: we refer for instance to  \cite{ManettiSeattle, LMDT} for a detailed treatment of the topic.

\subsection{A free graded Lie algebra which is not cofibrant}

In the model structure on DG-Lie algebras living in degrees $i < 0$ and $i \leq 0$
all DG-Lie algebras whose underlying graded Lie algebra  is free are cofibrant: see \cite[Proposition 5.5]{QuiR} and \cite[Proposition 5.7]{Nei} respectively. This example  shows that this is not true in the unbounded case.

Let $t$ be an indeterminate of degree $1$. Consider the free graded Lie algebra $\Li \langle t\rangle  $  generated by $t$: as a graded vector space it is generated by $t$ and by $[t,t]$, because $[t,[t,t]]=0$ by the Bianchi identity.
Since the differential of $t$ has to be of degree two, either $dt=0$ or $dt= \lambda [t,t]$ for some $0 \neq \lambda\in \K$.
Consider the case where the differential of $t$ is not zero: by rescaling $t$ we can suppose without loss of generality that $dt= -\frac{1}{2}[t,t]$, i.e., that $t$ is a Maurer-Cartan element. In fact, if we set $t':= -2 \lambda t$, then
\[ dt' = -2 \lambda dt= -2 \lambda^2 [t,t]= -2 \lambda^2 \left[-\frac{1}{ 2 \lambda}t', -\frac{1}{ 2 \lambda}t'\right]= - \frac{1}{2}[t',t'].\]

\begin{proposition}
	The free graded Lie algebra generated by one indeterminate in degree $1$ is cofibrant if and only if its differential is trivial. 
\end{proposition}

\begin{proof}
	Consider first the case where the differential is trivial, then the initial map
	\[ 0 \to \Li \langle t\rangle\]
	is an elementary semifree extension: in fact $dt=0 \subset 0$.
	Let now $dt$ be different from zero, then by the above considerations we can set without loss of generality $dt+ \frac{1}{2}[t,t]=0$.
	
	Notice that for any DG-Lie algebra $M$ and any Maurer-Cartan element $x \in M^1$, one has a morphism of DG-Lie algebras  
	\[ f_x \colon \Li \langle t\rangle\to M, \quad t \mapsto x,\]
	because $df_x(t)= dx= -\frac{1}{2} [x,x]= -\frac{1}{2} [f_x(t),f_x(t)]= \frac{1}{2} f_x[t,t]= f_x(dt).$
	
	Let $0 \to V \to A \xrightarrow{\varphi} B \to 0$ be any small extension of Artin local
	$\K$-algebras with residue field $\K$,
	and consider the following commutative diagram, where $L$ is a DG-Lie algebra, $x$ is a Maurer-Cartan element of $L \otimes \mathfrak{m}_B$ and $f_x(t)=x$.
	\begin{center}
		\begin{tikzcd}
			0 \arrow[d] \arrow[r]                       & L \otimes \mathfrak{m}_A \arrow[d, "{\varphi}"] \\
			\Li \langle t\rangle\arrow[r, "f_x"'] \arrow[ru, "h", dotted] & L \otimes \mathfrak{m}_B  .                              
		\end{tikzcd}
	\end{center}
	The vertical map on the right is surjective, and hence a fibration in the model structure on the category of DG-Lie algebras. Notice that 
	$H^*(\Li \langle t\rangle)=0$, so that the vertical map on the left is a weak equivalence. If we suppose that $\Li\langle t \rangle $ is cofibrant, by the lifting axiom of Definition~\ref{def.mod} there exists a map $h$ making both triangles commute: hence there exists $y=h(t) \in L\otimes \mathfrak{m}_A$ such that \[dy= dh(t)= h(dt)= h \left(-\frac{1}{2}[t,t]\right)=-\frac{1}{2}[h(t),h(t)]= -\frac{1}{2}[y,y],\]
	and that ${\varphi}(y)= {\varphi}(h(t))= f_x(t)=x$, i.e., there exists a Maurer-Cartan element of $L\otimes \mathfrak{m}_A$  lifting $x$.
	This would of course imply that any Maurer-Cartan functor is unobstructed, which is absurd. This fact is clear in the context of deformation theory, however we give here an easy example of a DG-Lie algebra whose Maurer-Cartan functor is obstructed:  the DG-Lie algebra  $(\Li \langle t\rangle , d\equiv 0)$ generated by one indeterminate in degree 1 and with trivial differential. In fact, the primary obstruction map, which is equal to the induced bracket in cohomology
	\[ [-,-] \colon H^1(L) \times H^1(L) \to H^2(L),\]
	is not trivial, so that the Maurer-Cartan functor is obstructed (see e.g. \cite[Appendix B]{LMDT}).   

\end{proof}

\subsection{The cobar construction for a locally conilpotent cocommutative coalgebra}

A coassociative graded coalgebra is a pair $(C,\Delta)$ consisting of
a graded vector space $C$ and a morphism of graded vector spaces $\Delta \colon C \to  C \otimes C$,
called a coproduct, satisfying the coassociativity equation
\[(\Delta \otimes \Id_C )\Delta = (\Id_C \otimes \Delta)\Delta \colon  C \to C \otimes  C \otimes C.\]
The iterated coproducts 
$ {\Delta}^n \colon {C} \to {C}^{\otimes n+1}$
are defined recursively for $n \geq 0 $ by the formulas
\[ {\Delta}^0 = \Id, \quad {\Delta}^n \colon {C} \xrightarrow{{\Delta}} {C}\otimes {C} \xrightarrow{\Id \otimes {\Delta}^{n-1}} {C} \otimes {C}^{\otimes n} = {C}^{\otimes n+1}.\]
The kernels of ${\Delta}^n$ form an increasing filtration: $\ker {\Delta}^n \subset \ker {\Delta}^{n+1}$ for every $n \geq  0$, see e.g. \cite{Hin2} or \cite[Chapter 11]{LMDT}.

\begin{lemma}\label{lem.coprodotti}
	Let $(C,\Delta)$ be a graded coalgebra, then for every $n \geq 1$,
	\[ \Delta (\ker \Delta^n) \subseteq \ker \Delta^{n-1} \otimes \ker \Delta^{n-1}.\]
\end{lemma}
See \cite[Lemma 11.1.10]{LMDT} for the proof. 

\begin{definition}
	 The graded coalgebra $(C, \Delta)$ is called (graded) cocommutative if
	$\tw \circ \Delta = \Delta$, where
	\[\tw \colon C \otimes  C \to C \otimes C, \quad 
	\tw(x \otimes y) = (-1)^{\overline{x} \  \overline{y}} y \otimes x,\]
	is the twist map.
\end{definition}

\begin{definition}
	The graded coalgebra $(C, \Delta)$ is called locally conilpotent if $C= \bigcup_n \ker \Delta^n$.
\end{definition}

This property is sometimes also called primitive cogeneration in the literature, see e.g. \cite[Chapter 22]{FHT}.

We briefly recall here the cobar construction for a cocommutative coalgebra. For more details we refer to \cite[Chapter 22]{FHT}, \cite[Appendix B]{QuiR} for the bounded case and to \cite{Hin2} for the unbounded case.
Let $(C,\Delta)$ be a locally conilpotent cocommutative graded coalgebra, and denote by $s \colon C \to C[-1]$ the suspension.  Consider the free graded algebra  $T(C[-1])$, which is a DG-algebra with differential defined as
\[ d(sx)= \sum (-1)^{\overline{x_i}} sx_i \otimes sx_i',\]
for $x \in C$ such that $\Delta(x)= \sum_i x_i \otimes x_i'$, and extended to the whole tensor algebra via the Leibniz rule. 
 By the cocommutativity of the coproduct $\Delta$, the above differential restricts to the free graded Lie algebra $\Li(C[-1])$; 
 the differential can be expressed as
\begin{equation}\label{eq.diffcobar}
	d(sx)= \frac{1}{2} \sum (-1)^{\overline{x_i}} [sx_i,sx_i'],
\end{equation}
for $x \in C$ such that $\Delta(x)= \sum_i x_i \otimes x_i'$.
 The cobar construction $\mathcal{L}(C)$ associated to $(C, \Delta)$ is the DG-Lie algebra  $\Li(C[-1])$ with the differential \eqref{eq.diffcobar}.

\begin{proposition}\label{prop.coalg}
	The cobar construction $\mathcal{L}(C)$ associated to the locally conilpotent cocommutative graded coalgebra $(C, \Delta)$ is a semifree DG-Lie algebra. 
\end{proposition}

\begin{proof}
	We show  that 
	the initial map $0 \to \mathcal{L}(C)$ is the countable composition of semifree extensions. Consider the increasing filtration of $C$ given by the kernels of the iterated coproduct:
	\[ 0= \ker \Delta^0 \subset \ker \Delta \subset \cdots \subset \ker \Delta^n \subset \ker \Delta^{n+1} \subset \cdots,\]
	and the associated quotients
	\[ A^n:= \frac{\ker \Delta^n}{\ker \Delta^{n-1} }.\] 
	Denoting by $B^n:= A^n[-1]$, then $C[-1]= \bigoplus_n B^n$, because by hypothesis $(C, \Delta)$ is locally conilpotent. 
	We prove that for every $n\geq 1$ the map
	\[  a_n \colon \coprod_{i=1}^{n-1}\Li (B^{i}) \to \coprod_{i=1}^n\Li(B^{i}),\]
	is an elementary semifree extension, i.e., that $d(B^n) \subseteq  \coprod_{i=1}^{n-1}\Li (B^{i})$.
	
We claim that
		for every $n \geq 1$, the differential $d$ descends to 
		\[ \overline{d} \colon B^n \to \Li (B^{n-1})_2 \subset B^{n-1}\otimes B^{n-1}.\]
	Fix $n \geq 1$; by Lemma~\ref{lem.coprodotti}, one has that
	${\Delta}\colon \ker {\Delta}^n \to  \ker {\Delta}^{n-1} \otimes \ker {\Delta}^{n-1}$, hence
		\[{\Delta}[-1]\colon \ker \Delta^n [-1] \to  \ker {\Delta}^{n-1} [-1]\otimes \ker {\Delta}^{n-1}[-1], \quad \forall n \geq 1.\]
			Since the differential on ${C}[-1]$ is defined  as $d= {\Delta}[-1]$, one obtains a map
			\[ \overline{d} \colon \frac{\ker {\Delta}^n}{\ker{\Delta}^{n-1}}[-1] \to  \frac{\ker {\Delta}^{n-1}}{\ker {\Delta}^{n-2}}[-1]  \otimes  \frac{\ker {\Delta}^{n-1}}{\ker {\Delta}^{n-2}}[-1]. \]
			By the cocommutativity of the coproduct ${\Delta}$, the image of $\overline{d}$ is contained inside $\Li (\frac{\ker {\Delta}^{n-1}}{\ker {\Delta}^{n-2}}[-1])_2= \Li (B^{n-1})_2$.

	The countable composition 
	\[ 0 \to \Li(B^1) \to \Li(B^1)\amalg \Li(B^2) \to \cdots \to \coprod_{i=1}^{n}\Li(B^{i}) \to \coprod_{i=1}^{n+1}\Li(B^{i}) \to \cdots \]
	is equal to the initial map $0 \to \mathcal{L}(C)$, because $\coprod_{i=1}^{n}\Li(B^{i}) \cong \Li (\bigoplus_{i=1}^n B^i)$ as graded Lie algebras. 
\end{proof}

\begin{remark}
	By the results of \cite{QuiR} in the bounded above case and \cite{Hin2} in the unbounded case, there exists a model structure on the category of differential graded cocommutative coalgebras. Moreover the cobar construction is a left Quillen functor, which means it preserves cofibrations and trivial cofibrations. Cofibrations in the category of DG-coalgebras are the injective maps, therefore in the setting of Proposition~\ref{prop.coalg} it is clear that $\mathcal{L}(C)$ is cofibrant, since all DG-coalgebras are cofibrant.
\end{remark}

\end{document}